\documentclass[10pt]{article}
\usepackage[all,2cell]{xy} \UseAllTwocells \SilentMatrices
\usepackage{latexsym,amsfonts,amssymb}
\usepackage{amsmath,amsthm,amscd}
\usepackage{hyperref,psfrag}
\usepackage{color}
\usepackage{etoolbox}
\usepackage[dvips]{epsfig}

\usepackage{graphicx}
\usepackage[centertableaux]{ytableau}
\usepackage{a4wide}
\usepackage{geometry}

\usepackage{epigraph,wrapfig}

\normalfont\upshape

\usepackage{fancyheadings}
\theoremstyle{definition}
\newtheorem{thm}{Theorem}[section]
\newtheorem{cor}[thm]{Corollary}

\newtheorem{lem}[thm]{Lemma}
\newtheorem{rem}[thm]{Remark}
\newtheorem{prop}[thm]{Proposition}
\newtheorem{defn}[thm]{Definition}

\numberwithin{equation}{section}

\usepackage{bbm}

\def\1{{\mathbbm{1}}}

\newcommand{\Hom}{{\rm Hom}}

\newcommand{\End}{{\rm End}}

\def\lra{{\longrightarrow}}
  
\def\udmod{{\mbox{-}\mathrm{\underline{mod}}}}  

\def\Ext{{\mathrm{Ext}}}

\def\shuffle{\,\raise 1pt\hbox{$\scriptscriptstyle\cup{\mskip
               -4mu}\cup$}\,}

\newcommand{\refequal}[1]{\xy {\ar@{=}^{#1}
(-1,0)*{};(1,0)*{}};
\endxy}

\newcommand{\dmod}{\mbox{-}\mathrm{mod}}

\title{
\hspace{5in} {\it {\normalsize To my wife Bai Xue}} \\
 \hspace{1in}  \\
Morphism spaces in stable categories of Frobenius algebras}
\author{You Qi}

\date{January 23, 2018}

\begin{document}
%

\maketitle

\begin{abstract}
We present an explicit formula computing morphism spaces in the stable category of a Frobenius algebra.
\end{abstract}

\setcounter{tocdepth}{2} 


\section{Introduction}
The stable categories of self-injective algebras are among the first appearances of triangulated categories besides the usual homotopy or derived categories. Their study has played fundamental roles in representation theory (finite group schemes) and algebraic geometry (singularity theory, matrix factorizations etc.). We refer the reader to the beautiful books \cite{Hap88, ZimRepBook} for an introduction and interesting applications.

Motivated by some recent application of stable categories of finite-dimensional Hopf algebras in categorification (see, e.g., \cite{Hopforoots, QYHopf, QiSussan2}, for motivations and some applications), one would like to have a better understanding of morphism spaces in stable categories. After all, the understanding of a category should be focused on its morphism spaces rather than only on objects. As an important class of examples, finite-dimensional Hopf algebras are Frobenius by a classical result of Larson and Sweedler \cite{LaSw}, and thus are self-injective. In \cite{QYHopf}, a formula describing the morphism spaces in the stable category of finite-dimensional Hopf algebras is given, utilizing the Hopf algebra action on morphism spaces between arbitrary modules. It is therefore a natural question to ask whether one can generalize the formula to arbitrary self-injective algebras. In this note, we provide an answer to this question when the self-injective algebra is \emph{Frobenius}. 

In a bit more detail, recall that for a self-injective algebra $A$, its stable module category, denoted $A\udmod$, is the quotient of the abelian $A$-modules by the ideal of morphisms that factor through projective-injective modules. A classical result that goes back to Heller \cite{Heller} tells us that $A\udmod$ is triangulated.

Also recall that (see, e.g., \cite{Kadison}) a Frobenius $A$ over a ground field $\Bbbk$ is characterized by a linear algebraic system $(\epsilon, \{a_i|i\in I\}, \{b_i|i\in I\})$ which we refer to as a \emph{Frobenius system}. Here $\epsilon:A\lra \Bbbk$ is a non-degenerate trace function and $ \{a_i\} $, $ \{b_i\}$ are dual bases of $A$ under the trace pairing. Frobenius algebras are self-injective. In terms of the Frobenius datum for $A$, we call an $A$-module map between two modules $M$ and $N$ to be \emph{null-homotopic}\footnote{Here the analogy is taken from the Frobenius algebra $\Bbbk[x]/(x^2)$ when $\mathrm{char}(\Bbbk)=2$. The dual bases are $\{1,x\}$ and $\{ x,1\}$ respectively, so that ``null-homotopic'' maps of the form $f=xh+hx$,  agreeing with the usual notion of null-homotopies.} if there is a $\Bbbk$-linear map $h:M\lra N$ such that 
\begin{equation}
f(m)=\sum_{i\in I}a_ih(b_im),
\end{equation}
for any $m\in M$.
Our main result (Theorem \ref{thm-main}) states that, for any two $A$-modules $M$ and $N$:
\begin{equation}
\Hom_{A\udmod}(M,N)=\dfrac{\Hom_A(M,N)}{\{\textrm{null-homotopic maps}\}}.
\end{equation}
By comparison with the usual homotopy formula for chain complexes, this formulation of the morphism space appears quite pleasing and fundamental in nature. The result might possibly be known to experts in the field (one can see implicit shadows of the formula from, e.g., \cite[Definition 1.6]{Broue} or \cite[Section 3]{LZZ}), nevertheless, the author could not find it explicitly stated in references, and thus would like to record the formula in this note. As a straightforward application, we provide a simple proof of Brou\'{e}'s Theorem characterizing the stable center \cite[Proposition 3.13]{Broue}, even when the Frobenius algebra is non-symmetric.

The main result (Theorem \ref{thm-main}) is presented in Section \ref{sec-Mor-space} after some basic notions are recalled in Section \ref{sec-Frob-alg}. In the last Section \ref{sec-apps}, we point out a few straightforward applications. Theorem \ref{thm-main} generalizes readily to super and/or graded Frobenius algebras. It also adapts itself to the more relative situation of Frobenius ring extensions as in \cite[Chapter 1]{Kadison} without essential changes. We expect that the characterization of stable morphism spaces in this note will prove useful in further applications.

\paragraph{Acknowledgments.} I would like to express my deepest appreciation to my wife, Ms.~Xue Bai, for her love, care and always being so supportive. It is my great pleasure to dedicate this cute little formula to her, on the occasion of our first anniversary.

The author is partially funded through the NSF research grant DMS-1763328.


\section{Basic definitions} \label{sec-Frob-alg}

\paragraph{Notation.} In this paper, we will let $\Bbbk$ denote a fixed ground field once and for all. Unadorned tensor product ``$\otimes$'' stands for tensor product over $\Bbbk$, and likewise $\Hom$ stands for the space of $\Bbbk$-linear homomorphisms. For any $\Bbbk$-algebra, we will denote by $A^e:=A\otimes A^{\mathrm{op}}$ be the \emph{enveloping algebra}, so that a left $A^e$-module is non other than an $(A,A)$-bimodule.

\paragraph{Stable categories of self-injective algebras.} Let $A$ be a self-injective algebra over a field $\Bbbk$. The stable category of $A$-modules, denoted by $A\udmod$, is the categorical quotient of $A\dmod$ by the class of projective-injective objects. More precisely, for any two $A$-modules $M,N$, let us denote the space of morphisms in $A\udmod$ factoring through projective injective $H$-modules by $\mathrm{I}(M,N)$. It is readily seen that, the collection of $\mathrm{I}(M,N)$'s ranging over all $M,N\in A\dmod$ constitutes an ideal in $A\dmod$. Then $A\udmod$ has the same objects as that of $A\dmod$, while the morphism space between two objects $M,N\in A\dmod$ is by definition the quotient
\begin{equation}
\Hom_{A\udmod}(M,N):=\dfrac{\Hom_{A\dmod}(M,N)}{\mathrm{I}(M,N)}.
\end{equation}
A classical Theorem of Heller states that $A\udmod$ is triangulated. The shift functor $[1]$ on $A\udmod$ is defined as follows. For any $M\in A\dmod$, choose an injective envelope $I_M$ for $M$ in $A\dmod$ and let $M^\prime$ be the cokernel of the natural injection:
\[
0\longrightarrow M\longrightarrow I_M\longrightarrow M^\prime\longrightarrow 0.
\]
Then $M[1]:= M^\prime$. The inverse functor $[-1]$ can be defined similarly by taking a projective cover and the corresponding kernel of the natural surjection map. We refer the reader to the references \cite{Hap88, ZimRepBook} for more details on this fundamental construction.

The main goal of this note is to point out an explicit description of the morphism space when the self-injective algebra arises as a Frobenius algebra. 

\paragraph{Frobenius algebras.}Let us recall the definition and basic properties of Frobenius algebras, following \cite{Kadison}.

\begin{defn}\label{def-Frob-system}
A $\Bbbk$-algebra $A$ is called \emph{Frobenius} if $A$ is equipped with a $\Bbbk$-linear \emph{trace function}
\[
\epsilon: A\lra \Bbbk
\]
and dually-paired bases $\{a_i|i=1,\dots, \mathrm{dim}(A)\}$ $\{b_i|i=1,\dots, \mathrm{dim}(A)\}$ satisfying
\[
\sum_{i}a_i\epsilon(b_ia)=a=\sum_{i}\epsilon(aa_i)b_i
\]
for any $a\in A$.
\end{defn}

We will refer to the datum of $(\epsilon, \{a_i\}, \{b_i\})$ for $A$ as a \emph{Frobenius system}.

\begin{rem}\label{rem-Frob-properties}
\begin{enumerate}
\item[(i)]One can easily show that the definition above is equivalent to the more classical definition of requiring that $\epsilon$ to be non-degenerate, in the sense that for any $a\in A$, there exists $b\in A$ such that $\epsilon(ab)=1$. Alternatively, this is equivalent to requiring the map
\[
A\lra A^*,\quad a\mapsto (a\cdot \epsilon:b\mapsto \epsilon(ba),\  \forall b\in A)
\]
\[
\left( \textrm{resp.}~A\lra A^*,\quad a\mapsto (\epsilon\cdot a: b\mapsto \epsilon(ab), \ \forall b\in A),\right)
\]
to be an isomorphism of left (resp.~right) $A$-modules.
\item[(ii)] Frobenius systems for $A$ constitute a torsor under the group action of (multiplicative) invertible elements of $A$. More precisely, if  $(\epsilon, \{a_i\}, \{b_i\})$ is Frobenius system for $A$ and $d\in A$ is invertible, then  $(d\cdot \epsilon , \{a_id^{-1}\}, \{b_i\})$ (or  $(\epsilon\cdot d, \{a_i\}, \{d^{-1}b_i\})$) are other Frobenius systems for $A$. The correspondence is a one-one bijection (\cite[Theorem 1.6]{Kadison}).
\item[(iii)] From the definition, it is clear that, if $A$, $B$ are Frobenius, then so are $A^\mathrm{op}$ and $A\otimes B$. In particular, $A^e=A\otimes A^{\mathrm{op}}$ is Frobenius. It is easy to see that a Frobenius system for $A^{e}$ can be taken to be $(\epsilon\otimes \epsilon, \{a_i\otimes b_j\}, \{b_i\otimes a_j\})$.
\end{enumerate}
\end{rem}

We recall the following definition from \cite[Corollary 1.5]{Kadison}.

\begin{defn}
Let $A$ be a Frobenius algebra with Frobenius system  $(\epsilon, \{a_i\}, \{b_i\})$. The \emph{Frobenius element} for $A$ is defined to be
\[
\xi_A:=\sum_{i}a_i\otimes b_i\in A\otimes A^\mathrm{op}.
\]
\end{defn}

\begin{lem}\label{lem-canonical-bimod}
The Frobenius element satisfies 
\[
\sum_{i}aa_i\otimes b_i=\sum_i a_i\otimes b_ia
\]
for all $a\in A$. Consequently, the map
\[
A\cong \xi_A A^e ,\quad a\mapsto \xi_Aa=a\xi_A
\]
is an isomorphism of $(A,A)$-bimodules.
\end{lem}
\begin{proof}
See \cite[Corollary 1.5]{Kadison}.
\end{proof}


\section{Morphism spaces} \label{sec-Mor-space}

If $M$ is any $A$-module, let us denote by $M_0$ the underline $\Bbbk$-vector space of $M$. Likewise, if $f:M\lra N$ is an $A$-linear map, we will denote by $f_0: M_0\lra N_0$ the corresponding $\Bbbk$-vector space map.

\begin{lem}\label{lem-canonical-embedding}
Let $A$ be a Frobenius algebra with its Frobenius system  $(\epsilon,\{a_i\},\{b_i\})$. If $M$ is any $A$-module, there is a canonical embedding of left $A$-modules
\[
\phi_M: M\lra A\otimes M_0, \quad m\mapsto \sum_i a_i\otimes b_i m.
\]
\end{lem}
\begin{proof}
This follows from the fact that the Frobenius element $\xi_A=\sum_i a_i\otimes b_i$ satisfies
$a\xi_A=\xi_Aa$ for all $a\in A$. The map is an injection since, by Definition \ref{def-Frob-system}, there is a $\Bbbk$-linear splitting map $A\otimes M_0\lra M$ defined by $a\otimes m\mapsto \epsilon(a)m$:
\[
\sum_ia_i\otimes b_im\mapsto \epsilon(a_i)b_im=m. \qedhere
\]
\end{proof}

Despite the fact that the above embedding is usually larger than the injective envelope of $M$, the apparent functoriality in the above construction makes computing the morphism space in $A\udmod$ a lot easier than in, say, that of  an arbitrary non-Frobenius self-injective algebra. 

\begin{lem}\label{lem-factorization}
Let $A$ be a Frobenius algebra, and $f:M\lra N$ be a left $A$-module homomorphism. Then $f$ factors through a projective-injective $A$-module if and only if $f$ is equal to the composition of $A$-module maps
\[
f: M\xrightarrow{\phi_M} A\otimes M_0 \stackrel{h}{\lra} N,
\]
where $h$ is some $A$-linear homomorphism.
\end{lem}
\begin{proof}
It suffices to prove the lemma when $N$ is a projective-injective $A$-module. Consider the commutative diagram
\[
\begin{gathered}
\xymatrix{
M \ar[r]^f \ar[d]_{\phi_M}& N \ar[d]^{\phi_N}\\
A\otimes M_0 \ar[r]^{\mathrm{Id}_A\otimes {f_0}} & A\otimes N_0
}
\end{gathered} \ ,
\]
where $\phi_M$ and $\phi_N$ are the canonical maps as in the previous Lemma. Since $N$ is injective, the canonical injection $\phi_N$ admits a section $g: A\otimes N_0 \lra N$. Therefore $f$ factors through as
\[
\xymatrix{
f: M\ar[r]^-{\phi_M} & A\otimes M_0 \ar[r]^{\mathrm{Id}_A\otimes {f_0}} & A\otimes N_0 \ar[r]^-g & N.
}
\]
The result follows.
\end{proof}

Let $M,~N$ be two $A$-modules. The space of $\Bbbk$-linear maps $\Hom(M,N)$ has a natural $(A,A)$-bimodule structure by declaring, for an $a\otimes b\in A^e=A\otimes A^{\mathrm{op}}$, $\phi\in \Hom(M,N)$ and $m\in M$, that
\begin{equation}\label{eqn-bimod-Hom}
((a\otimes b) \cdot \phi)(m):=a\phi(bm).
\end{equation} 
In particular, we can apply the Frobenius element to any linear map in $\Hom(M,N)$. 

\begin{lem}\label{lem-trivial-invariants}
For any linear map $\phi \in \Hom(M,N)$, 
\[
(\xi_A\cdot \phi):=\sum_i (a_i\otimes b_i)\cdot \phi \in \Hom_A(M,N).
\]
\end{lem}
\begin{proof}
By Lemma \ref{lem-canonical-embedding}, $\sum_i aa_i\otimes b_i=\sum_{i}a_i\otimes b_ia$ for any $a\in A$. Applying this equation to $\phi$, we obtain that the equality
\[
(\sum_i aa_i \phi(b_im))=\sum_{i}(a_i\phi(b_iam))
\]
holds for all $m\in M$. The result follows.
\end{proof}

\begin{prop}\label{prop-characterization}
Let $A$ be a Frobenius algebra as above. 
An $A$-module map $f:M\lra N$ factors through the canonical $A$-module map $\phi_M:M\lra A\otimes M_0$ if and only if there is a $\Bbbk$-linear map $\phi: M_0 \lra N_0$ such that, for any $m\in M$,
\[
f(m)=\sum_i a_i\phi(b_im)
\]
\end{prop}
\begin{proof}
First off, suppose $f:M\lra N$ factors through $\phi_M$ as
\[
f: M\xrightarrow{\phi_M} A\otimes M_0 \stackrel{\psi}{\lra} N
\]
for some $A$-linear map $\psi$.
Setting $\phi$ to be the $\Bbbk$-linear map $\psi|_{1\otimes M_0}:{M_0}\lra N$, we have
\[
f(m)=\psi\circ \phi_M(m)=\psi(\sum_ia_i\otimes b_im)=\sum_ia_i\psi(1\otimes b_im)=\sum_ia_i\phi (b_im).
\]
Here in the third equality, we have used that $\psi$ is $A$-linear.

Conversely, if $f$ arises as $f(m):=\sum_{i}a_i\phi(b_im)$ for some $\phi:M_0\lra N_0$, then $f$ is $A$-linear by the previous Lemma. Define
\[
\psi: A\otimes M_0 \lra N ,\quad
\psi (a\otimes m):=a\phi(m)
\]
This is clearly $A$-linear, and, for any $m\in M$,
\[
\psi\circ \phi_M(m)=\sum_i\psi(a_i\otimes b_im)=\sum_ia_i\phi(b_im)=f(m),
\]
giving us the desired factorization of $f$.
\end{proof}

For any $A^e$-module $H$, we define the space of \emph{$A$-invariants} to be
\begin{equation}
H^A:=\Hom_{A^e}(A,H).
\end{equation} 
Since $\Hom(M,N)$ is an $A^e$-module (equation \eqref{eqn-bimod-Hom}), the space of $A$-invaraints in $\Hom(M,N)$ equals
\begin{equation}\label{eqn-Hom-A-as-invariants}
\Hom(M,N)^A:= \Hom_{A^e}(A,\Hom(M,N)).
\end{equation} 
Any element of $\Hom_{A^e}(A,\Hom(M,N))$ sends the central unit element $1_A$ into some $f\in \Hom(M,N)$ satisfying, for any $a\in A$, that $fa=af$, i.e., $f$ is an $A$-linear map. Conversely, any $A$-linear $f$ determines an $(A,A)$-bimodule map from $A$ to $\Hom(M,N)$, sending $a\in A$ to $af(\mbox{-})=f(a\mbox{-})=fa(\mbox{-})$.  It follows that
\begin{equation}
\Hom(M,N)^A\cong\Hom_{A^e}(A, \Hom(M,N))\cong \Hom_A(M,N).
\end{equation}
By Lemma \ref{lem-trivial-invariants}, it is clear that
\[
\xi_A\cdot \Hom(M,N)\subset \Hom(M,N)^A,
\]
so that the quotient $\frac{\Hom_A(M,N)}{\xi_A\cdot \Hom(M,N)}$ makes sense.

\begin{thm}\label{thm-main}
Let $A$ be a Frobenius algebra with a Frobenius system $(\epsilon,\{a_i\}, \{b_i\})$, and let $M$, $N$ be two left $A$-modules. Then there is an isomorphism of $\Bbbk$-vector spaces
\[
\Hom_{A\udmod}(M,N)\cong \dfrac{\Hom_A(M,N)}{\xi_A\cdot \Hom(M,N)},
\]
which is natural in $M$ and $N$.
\end{thm}
\begin{proof}
Lemma \ref{lem-trivial-invariants} shows that $\xi_A\cdot \Hom(M,N)\subset \Hom_A(M,N)$, and the quotient on the right hand side is well-defined.

By definition, the space of morphisms between $M$ and $N$ in the stable category $A\udmod$ is the quotient of $\Hom_A(M,N)$ by the ideal $\mathrm{I}(M,N)$ consisting of morphisms that factor through projective-injective modules. Via Lemma \ref{lem-factorization}, this ideal coincides with
\[
\mathrm{I}(M,N)=\left\{f|f=\psi\circ\phi_M,~ \psi:A\otimes M_0\lra N\right\}.
\] 
Now Proposition \ref{prop-characterization} shows that $\mathrm{I}(M,N)= \xi_A\cdot \Hom(M,N)$.

Finally, if $g: M^\prime \lra M$ is an $A$-module map, then
\[
g^*:\Hom(M,N)\lra \Hom(M^\prime, N)
\]
is an $(A,A)$-bimodule homomorphism. Similar statements hold for an $A$-module map $h:N\lra N^\prime$. The naturality follows.
\end{proof}

We remark that Theorem \ref{thm-main} also gives a way to compute arbitrary $\Ext$-spaces in $A\udmod$, as follows. Choose an injective embedding for $N$ as in Lemma \ref{lem-canonical-embedding}, so that the quotient is denoted $N^\prime$:
\[
0\lra N\lra A\otimes N_0\lra N^\prime \lra 0.
\]
Then 
\begin{equation}
\Ext_{A\udmod}^1(M,N)=\Hom_{A\udmod}(M,N[1])=\Hom_{A\udmod}(M,N^\prime).
\end{equation}
Continuing the process for $N^\prime$ and iterating gives us all positive $\Ext$-spaces between $M$ and $N$. Likewise, performing a similar construction for $M$ gives rise to all negative $\Ext$-spaces.

We can further compare the stable morphism space $\Hom_{A\udmod}(M,N)$ with another stable morphism space, but rather in the stable category of the enveloping algebra $A^e$.

\begin{cor}
Let $A$ be a Frobenius algebra, and $M,N$ be two $A$-modules. Then there is an isomorphism of $\Hom$-spaces
\[
\Hom_{A\udmod}(M,N)=\Hom_{A^e\udmod}(A,\Hom(M,N)),
\]
which is natural in $M$ and $N$.
\end{cor}
\begin{proof}
By Remark \ref{rem-Frob-properties} (iii), The algebra $A^e$ is Frobenius with a Frobenius system $(\epsilon\otimes \epsilon,~ \{a_i\otimes b_j\}, ~\{b_i\otimes a_j\})$. Therefore, by Theorem \ref{thm-main}, 
\[
\Hom_{A^e\udmod}(A,\Hom(M,N))=\dfrac{\Hom_{A^e}(A,\Hom(M,N))}{\xi_{A^e}\cdot \Hom(A,\Hom(M,N))}.
\]
By the discussion prior to the proposition, the numerator in the above expression agrees with $\Hom_A(M,N)$. Thus we are reduced to showing that the expression in the denominator agrees with the space
\begin{equation}\label{eqn-compare-1}
\xi_A\cdot \Hom(M,N)=\left\{f(m)=\sum_ia_ig(b_im) ,~g\in \Hom(M,N)\right\}.
\end{equation}

The Frobenius element for $A^e$ equals
\[
\xi_{A^e}=\sum_{i,j}a_i\otimes b_j\otimes b_i\otimes a_j\in A^e\otimes (A^e)^{\mathrm{op}}=A\otimes A^{\mathrm{op}}\otimes A^{\mathrm{op}}\otimes A.
\]
If $h\in \Hom(A,\Hom(M,N))$, then, for any $a\in A$ and $m\in M$, we have
\[
(\xi_{A^e}h)(a)(m)=\sum_{i,j}\left( (a_i\otimes b_j)h(b_iaa_j) \right)(m)=\sum_{i,j}a_i(h(b_iaa_j))(b_jm).
\]
Using Lemma \ref{lem-canonical-bimod}, the last term can be identified with
\[
\sum_{i,j}a_i(h(b_iaa_j))(b_jm)=\sum_{i,j}aa_i(h(b_ia_j))(b_jm),
\]
which in turn shows that $\xi_{A^e}h:A\lra \Hom(M,A)$ is left $A$-linear. Taking $a=1$, we obtain
\begin{equation}\label{eqn-compare-2}
(\xi_{A^e}h)(1)(m)=\sum_{i,j}a_ih(b_ia_j)(b_jm)= \sum_{i,j}a_ja_ih(b_i)(b_jm)=\sum_j a_j (\sum_i(a_i h(b_i)(b_jm))),
\end{equation}
where we have used Lemma \ref{lem-canonical-bimod} again in the second equality.

Comparing the expressions in equations \eqref{eqn-compare-1} and \eqref{eqn-compare-2}, it suffices to show that any linear map $g\in \Hom(M,N)$ arises as
$g=\sum_i a_i h_g(b_i)$
for some $h_g\in \Hom(A,\Hom(M,N))$. This can be explicitly done by taking
\[
h_g(a)=\epsilon(a)g,
\]
for then we will have
\[
\sum_{i}a_ih_g(b_i)=\sum_i a_i\epsilon(b_i)g=g.
\]
Here we have used the defining property of a Frobenius system (Definition \ref{def-Frob-system}, taking $a=1$) in the last equality. The proposition follows.
\end{proof}

Before we end this discussion, let us point that the above morphism space is an analogue of the classical (degree zero) Tate cohomology for a finite group $G$ with coefficients in a $G$-module. To do this, let us identify
\[
\xi_{A}\cdot \Hom(M,N)\cong \xi_A A^e  \otimes_{A^e} \Hom(M,N) \cong A\otimes_{A^e} \Hom(M,N),
\]
the last equality following from Lemma \ref{lem-canonical-bimod}. Thus we may identify
\begin{equation}
\Hom_{A\udmod}(M,N)\cong \dfrac{\Hom_{A^e}(A,\Hom(M,N))}{A\otimes_{A^e}\Hom(M,N)}.
\end{equation}
This is an analogue of the degree zero Tate cohomology group because, if $A=\Bbbk G$ is the group algebra of a finite group $G$, then, with the embedding
\[
\Bbbk G\lra \Bbbk G\otimes \Bbbk G^{\mathrm{op}}, \quad \quad  h\mapsto h\sum_{g}g\otimes g^{-1},
\]
we have
\[
A\otimes_{A^e} \Hom(M,N)=\left\{f:M\lra N|\exists h\in \Hom(M,N),~f(m)=\sum_{g\in G}gh(g^{-1}m)\right\}.
\]
Thus Theorem \ref{thm-main} recovers the degree zero Tate cohomology of the $G$-module $\Hom(M,N)$:
\begin{equation}
\Hom_{\Bbbk G\udmod}(M,N)=\dfrac{\Hom(M,N)^G}{\Hom(M,N)_G},
\end{equation}
the space of invariant quotient its subspace of coinvariants.


\section{Some applications} \label{sec-apps}

\paragraph{Finite-dimensional Hopf algebras.} By a classical Theorem of Larson and Sweedler \cite{LaSw}, finite-dimensional Hopf algebras are Frobenius. In \cite[Section 5]{QYHopf}, we have presented a morphism space formula for the stable category. We now recall the formula and point out its the relationship with Theorem \ref{thm-main}. 

Consider a finite-dimensional Hopf algebra $H$. Denote its comultiplication by $\Delta:H\lra H\otimes H$, counit by $\epsilon:H\lra \Bbbk$ and antipode by $S:H\lra H^{\mathrm{op}}$. Given any $h\in H$ we will follow Sweedler and write 
\begin{equation}
\Delta(h)=\sum_h h_1\otimes h_2.
\end{equation}
The morphism space of any two $H$-modules $M,N$ is equipped with a natural left $H$-module structure: for any $h\in H$, $\phi\in \Hom_\Bbbk(M,N)$, the map $h\cdot \phi$ is determined by the formula
\begin{equation}
(h\cdot \phi)(m)=\sum_h h_2\phi(S^{-1}(h_1)m)
\end{equation}
for any $m\in M$.

Let $\Lambda\in H$ be a non-zero left integral over $\Bbbk$, which is, up to rescaling, uniquely characterized by the property
\begin{equation}
h\Lambda=\epsilon(h)\Lambda
\end{equation}
for any $h\in H$. Then, given any $H$-modules $M$ and $N$, we have exhibited, in \cite[Corollary 5.7]{QYHopf}, the morphism space formula in $H\udmod$:
\begin{equation}
\Hom_{H\udmod}(M,N)=\dfrac{\Hom(M,N)^H}{\Lambda\cdot \Hom(M,N)}.
\end{equation}
This generalizes the identification of morphism spaces with the degree zero Tate cohomology of finite groups (see the last part of the previous section) to arbitrary finite-dimensional Hopf algebras. 

In order to match this formula with Theorem \ref{thm-main}, it suffices to note that, for a finite-dimensional Hopf algebra, the datum $(t, \{ h_2\}, \{S^{-1}(h_1)\} )$ constitutes a Frobenius system for $H$, where $t\in H^*$ is a \emph{left integral} on $H$ (i.e., an element of $H^*$ satisfying, for all $g\in H^*$, that $gt=g(1_H)t$). This can be found, for instance, in \cite[Proposition 6.4]{Kadison} , or the reader may try to prove it directly by definitions and axioms of Hopf algebras.

\paragraph{Stable center.} As a simple application, we reprove a Theorem of Brou\'{e} characterizing the \emph{stable center} of a Frobenius algebra $A$, which is first established in \cite{Broue} when $A$ is symmetric. Recall that the stable center of $A$ (see also \cite[Definition 5.9.1]{ZimRepBook}) equals
\begin{equation}
\widehat{Z}(A)=\End_{A^e\udmod}(A).
\end{equation}

We apply Theorem \ref{thm-main} to explicitly compute the stable center, for general Frobenius algebras that are not necessarily symmetric. Let $A$ be a Frobenius algebra with a Frobenius system $(\epsilon,~\{a_i\},~\{b_i\})$. Then
\begin{equation}
\widehat{Z}(A)=\dfrac{\Hom_{A^e}(A,A)}{\xi_{A^e}\cdot \Hom(A,A)}.
\end{equation}
The space $\Hom_{A^e}(A,A)$ is isomorphic to the usual center $Z(A)$ of $A$ by sending any $(A,A)$-bilinear map $f:A\lra A$ to the element $f(1)$.

Now, if $h\in \Hom(A,A)$, we have
\begin{equation*}
(\xi_{A^e}\cdot h)(a)=\sum_{i,j\in I}(a_i\otimes b_j)h((b_i\otimes a_j)\cdot a) =\sum_{i,j\in I}a_ih(b_iaa_j)b_j.
\end{equation*}
Such a map, when evaluated on $a=1$, equals
\[
(\xi_{A^e}\cdot h)(1)=\sum_{i,j\in I}a_ih(b_ia_j)b_j =\sum_{j} a_j\left(\sum_{i}a_ih(b_i)\right)b_j,
\]
where we have used Lemma \ref{lem-canonical-bimod} in the last step. It is thus clear that, under the identification of $\Hom_{A^e}(A,A)\cong Z(A)$, 
\[
\left\{\xi_{A^e}\cdot h\right\}\cong \left\{\sum_{j}a_j\left(\sum_ia_ih(b_i)\right)b_j\right\}\subset \left\{\sum_{j}a_jab_j\big|a\in A \right\}.
\]
We claim that the inclusion above is actually an identity. Indeed, it suffices to show that, for any $a\in A$, there exists an $h_a\in \Hom(A,A)$ such that $\sum_ia_ih_a(b_i)=a$. Take $h_a:=\epsilon(\textrm{-})a$. Then, we have
\[
\sum_i a_ih_a(b_i)=(\sum_ia_i\epsilon(b_i))a = a,
\]
the last equality coming from Definition \ref{def-Frob-system}. The claim follows.

Denote by $\xi_A\cdot a:=\sum_ia_iab_i$, and similarly $\xi_A\cdot A=\{\xi_A\cdot a|a\in A\}$. Remark \ref{rem-Frob-properties} (ii) shows that this subspace is independent of choices of $\xi_A$. In conclusion, we have established the following.

\begin{cor}[Brou{\'e}'s Theorem]\label{thm-stable-center}
There is an isomorphism of commutative algebras
\[
\widehat{Z}(A)=\dfrac{Z(A)}{\xi_{A}\cdot A}.
\]
\end{cor}
\begin{proof}
By the above discussion, it only suffices to show that $\xi_A\cdot A$ is an ideal of $Z(A)$. This is clear, since for any $z\in Z(A)$ and $a\in A$, we have
\[
z(\sum_ia_i a b_i)=\sum_ia_i (za) b_i= \sum_ia_i (az) b_i = (\sum_ia_i a b_i)z.
\]
The result follows.
\end{proof}

If $A=\Bbbk G$ is the group algebra of a finite group, then the Theorem implies that its stable center equals
the span of (conjugacy) class functions whose order of stablizer group does not divide the characteristic of the ground field.

\paragraph{The stable category of $\Bbbk[x]/(x^{n})$.} Let us consider this specific example when $A=\Bbbk[x]/(x^{n})$. The stable category $A\udmod$ is triangulated equivalent to the category of matrix factorizations over $\Bbbk[x]$ with potential $x^{n}$. The latter realization, when appropriately graded, is fundamental to the construction of Khovanov-Rozansky homology groups of knots and links \cite{KR1}.

To apply the main theorem to this example, we use that $A$ has the Frobenius trace 
\[
\epsilon (x^i)=
\left\{
\begin{array}{cc}
0 & i\neq n-1\\
1 & i = n-1
\end{array}
\right.
\]
and the ordered bases $\{1,x,\dots, x^{n-1}\}$, $\{x^{n-1},x^{n-2},\dots, 1\}$ that are dual to each other under the pairing. It follows that a morphism of $A$-modules $f:M\lra N$ is \emph{null-homotopic} (c.f.~the footnote of Introduction) if and only if there is a linear map $h:M\lra N$ such that
\begin{equation}\label{eqn-null-homotopy}
f(m)=\sum_{i=0}^{n-1} x^ih(x^{n-1-i}m)
\end{equation}
for all $m\in M$. Now any module in $A\udmod$ is a direct sum of indecomposables, the latter classified, up to isomorphism, by quotients of $A$ of the form $V_i:=\Bbbk[x]/(x^{i+1})$, $i=0,\dots, n-1$. The endomorphism algebra of $V_i$ in $A\udmod$ is a quotient of $\End_{A}(V_i)\cong \Bbbk[x^{i+1}]$ by the ideal of null-homotopic morphisms. It is easy to identify this ideal using equation \eqref{eqn-null-homotopy}, which equals $(x^{n-1-i})$ if $n-1-i\leq i$, or $0$ if $n-1-i\geq i+1$. Therefore, we obtain
\begin{equation}
\End_{A\udmod}(V_i)\cong 
\left\{
\begin{array}{cc}
\Bbbk[x]/(x^{n-1-i}) & \textrm{ if } i\geq (n-1)/2,\\
\Bbbk[x]/(x^{i+1}) & \textrm{ if } i\leq (n-2)/2.
\end{array}
\right.
\end{equation}

We next turn to identify the stable center of the category. Applying Theorem \ref{thm-stable-center}, we deduce that the size of the stable center of $A\udmod$ depends on the characteristic of $\Bbbk$ as follows. A simple computation shows that
\[
\xi_A\cdot A=\left\{\sum_{i=0}^{n-1} x^i a x^{n-1-i}\big|a=a_0+a_1x+\dots+a_{n-1}x^{n-1}\in A\right\}=\left\{a_0nx^{n-1}\big|a_0\in \Bbbk\right\}.
\]
It follows that $\xi_A\cdot A$ is trivial if and only if $\mathrm{char}(\Bbbk)|n$. We have thus established the following.

\begin{cor} 
The stable center of the Frobenius algebra $\Bbbk[x]/(x^{n})$ equals
\[
\begin{gathered}
\widehat{Z}(\Bbbk[x]/(x^{n}))
=\dfrac{\Bbbk[x]/(x^{n+1})}{(nx^{n-1})}
=\left\{
\begin{array}{cc}
\dfrac{\Bbbk[x]}{(x^{n-1})} & \mathrm{char}(\Bbbk)\nmid n,\\
\dfrac{\Bbbk[x]}{(x^{n})} & \mathrm{char}(\Bbbk)~|~n.
\end{array}
\right.
\end{gathered}
\]
\end{cor}
This example is also discussed in \cite[Example 5.9.8]{ZimRepBook} using more sophisticated methods.

\addcontentsline{toc}{section}{References}



%

\vspace{0.1in}

\noindent Y.~Q.: { \sl \small Department of Mathematics, California Institute of Technology, Pasadena, CA 91125, USA} \newline \noindent {\tt \small email: youqi@caltech.edu}

%
\end{document}